\documentclass[11pt]{amsart}
\usepackage{graphicx}
\usepackage[active]{srcltx}
 \makeatletter
\renewcommand*\subjclass[2][2000]{%
  \def\@subjclass{#2}%
  \@ifundefined{subjclassname@#1}{%
    \ClassWarning{\@classname}{Unknown edition (#1) of Mathematics
      Subject Classification; using '1991'.}%
  }{%
    \@xp\let\@xp\subjclassname\csname subjclassname@#1\endcsname
  }%
}
 \makeatother
\usepackage{enumerate,url,amssymb,  mathrsfs,yhmath}

\newtheorem{theorem}{Theorem}[section]
\newtheorem{lemma}[theorem]{Lemma}
\newtheorem*{lemma*}{Lemma}
\newtheorem{proposition}[theorem]{Proposition}
\newtheorem{corollary}[theorem]{Corollary}

\theoremstyle{definition}

\theoremstyle{remark}
\newtheorem{remark}[theorem]{Remark}

\numberwithin{equation}{section}

\newcommand{\abs}[1]{\lvert#1\rvert}

\DeclareMathOperator{\sign}{sign}

\def\XXint#1#2#3{{\setbox0=\hbox{$#1{#2#3}{\int}$}
\vcenter{\hbox{$#2#3$}}\kern-.5\wd0}}

\def\le{\leqslant}
\def\ge{\geqslant}
\begin{document}

\title[{Optimal estimates for the gradient of harmonic functions} ]{Optimal estimates for the gradient of
harmonic functions in the unit disk} \subjclass{Primary 31A05;
Secondary 42B30 }


\author{David Kalaj}
\address{University of Montenegro, Faculty of Natural Sciences and
Mathematics, Cetinjski put b.b. 81000 Podgorica, Montenegro}
\email{davidk@ac.me}

\author{Marijan Markovi\'c}
\address{University of Montenegro, Faculty of Natural Sciences and
Mathematics, Cetinjski put b.b. 81000 Podgorica, Montenegro}
\email{marijanmmarkovic@gmail.com}

\begin{abstract} Let $\mathbf U$ be the unit disk, $p\ge 1$ and let $h^p(\mathbf U)$ be the Hardy space of complex harmonic functions. We find the sharp constants $C_p$ and the sharp functions
$C_p=C_p(z)$ in the inequality
$$|Dw (z)|\leq {C_p}(1-|z|^2)^{-1-1/p}\|w\|_{h^p(\mathbf U)}, w\in h^p(\mathbf U), z\in \mathbf U,$$
in terms of Gauss hypergeometric and Euler functions.  This
generalizes some results of Colonna related to the Bloch constant of
harmonic mappings of the unit disk into itself and improves some
classical inequalities by Macintyre and Rogosinski.
\end{abstract}

\maketitle


\section{Introduction and  statement of the results}
 A harmonic function $w$ defined in the unit ball $\mathbf B^n$
belongs to the harmonic Hardy class $h^p=h^p(\mathbf  B^n)$, $1\leq
p<\infty$ if the following growth condition is satisfied
\begin{equation}\label{ee}\Vert w\Vert_{h^p}:
=\left(\sup_{0<r<1}\int_S |w(r\zeta)|^p d\sigma
(\zeta)\right)^{1/p}<\infty\end{equation} where $S=S^{n-1}$ is the
unit sphere in $R^n$ and $\sigma$ is the unique normalized rotation
invariant Borel measure on $S$. The space $h^{\infty}(\mathbf  B^n)$
contains bounded harmonic functions.

It turns out that if $w\in h^p(B^n)$, then there exists the finite
radial limit
$$\lim_{r\to 1^-}w(r\zeta) = f(\zeta)\ (\text{a.e. on } S
)$$ and the boundary function $f(\zeta)$ belong to the space
$L^p(S)$ of $p$-integrable functions on the sphere.

It is well known that harmonic functions from Hardy class can be
represented as Poisson integral
$$u(x)=\int_SP(x,\zeta)d\mu(\zeta), x \in B^n$$
where $$P(x,\zeta)=\frac{1-|x|^2}{ |x-\zeta|^n}, x\in B^n, \zeta\in
S$$ is Poisson kernel and $\mu$ is complex Borel measure. In the
case $p>1$ this measure is absolutely continuous with respect to
$\sigma$ and $d\mu(\zeta)=f(\zeta)d\sigma$. Moreover $$\Vert
w\Vert_{h^p}=\|\mu\|$$ and for $1<p\le\infty$ we have
\begin{equation}\label{more}\Vert w\Vert_{h^p}=\|\mu\|=\|f\|_p.\end{equation}
where we denote by $\|\mu\|$ total variation of the measure $\mu$.

For previous facts we refer to the book  \cite[Chapter~6]{ABR}.

For $n=2$ we use the classical notation  $\mathbf U$ and  $\mathbf T
$ to denote the unit disk in the complex plane $\mathbf C$ and its
boundary.

Let $L^p(\mathbf R^n)$ be the space of Lebesgue integrable functions
defined in $\mathbf R^n$ with the norm $$\|f\|_p=\left(\int_{\mathbf
R^n}|f(x')|^p dx'\right)^{1/p}.$$ Let $\omega_n$ be the area of the
unit sphere in $\mathbf R^n$. Let in addition $h^p(\mathbf R^n_+)$
be the Hardy space of real harmonic functions in $\mathbf R^n_+$,
which can be represented as the Poisson integral
$$u(x)=\frac{2}{\omega_n}\int_{\mathbf
R^n}\frac{x_n}{|y-x|^n}u(y')dy',$$ with boundary values in
$L^p(\mathbf R^{n-1})$, where $y=(y',0)$, $y'\in \mathbf R^{n-1}$.

In the recent paper \cite{Mar}  Maz'ya and Kresin studied point-wise
estimates of the gradient of real harmonic function $u$ under the
assumptions that the boundary values belong to $L^p$. They obtained
the following result $$|\nabla u(x)|\le C_px_n^{(1-n-p)/p}\|u\|_p$$
where $C_p$ is a constant depending only on $p$ and $n$. For $p=1$,
$p=2$ and $p=\infty$ the constant $C_p$ is concretized and it is
shown the sharpness of the result. After that, in \cite{Mar1}, they
obtained similar results for the unit ball, but for $p=1$ and $p=2$
only. Precisely, they obtain some integral representation for the
sharp constant $K_p(x,l)$ in the inequality
$$|\left<\nabla u(x),l\right>|\leq K_p(x,l)\|u\|_p, \ \ 1\le p\le \infty$$
and the sharp constant $K_p(x)$ in $$|\nabla u(x)|\leq
K_p(x)\|u\|_p$$ is concretized for $p=1,2$ and $x$ arbitrary and for
$x=0$ and all $p$.

Notice that, for $n=2$ the results concerning the upper half-plane
$\mathbf {H}$ cannot be directly translated to the unit disk and
vice-versa. Although the unit disk $\mathbf{U}$ and the upper
half-plane $\,\mathbf{H}\,$ can be mapped to one-another by means of
M\"obius transformations, they are not interchangeable as domains
for Hardy spaces. Contributing to this difference is the fact that
the unit circle has finite (one-dimensional) Lebesgue measure while
the real line does not.

A complex harmonic function  $w$ in a region $D$ can be expressed as
$w=u+iv$ where $u$ and $v$ are real harmonic functions in $D$. For a
complex harmonic function we will use sometimes the abbreviation a
harmonic mapping. If $D$ is simply-connected, then there are two
analytic functions $h$ and $k$ defined on $D$ such that
$w=g+\overline h.$ For a complex harmonic function $w=g+\overline{
h}=u+i v$, denote by $D w(z)$ the formal differential matrix $D
w(z)=\begin{pmatrix}
  u_{x} & u_{y} \\
  v_{x} & v_{y}
\end{pmatrix}$. Its norm is given by \begin{equation*}|D
w|:=\max\{|Dw(z)l|:|l|=1\}.\end{equation*} Then
\begin{equation}\label{opernorm}|Dw(z)|=|g'(z)|+|h'(z)|.\end{equation}

Let $w$ be a harmonic function satisfying the Lipschitz condition,
when regarded as a function from the hyperbolic unit disk into the
complex plane $\mathbf C$ endowed with the Euclidean distance.  The
function $w$ is called \emph{Bloch} with the \emph{Bloch constants}
$$\beta_w=\sup_{z\neq z'}\frac{|w(z)-w(z')|}{d_h(z,z')}.$$ Here
$d_h$ is defined by $$ \tanh \frac{d_h(z,z')}{2}= \frac{|z-z'|}{|1-z
\overline{z'}|} \,.
$$
It can be proved that
\begin{equation}\label{isi}\beta_w=\sup_{z\in
\mathbf U}(1-|z|^2)|Dw(z)|.\end{equation} We refer to
\cite[Theorem~1]{cor} for the proof of \eqref{isi}. In the same
paper Colonna proved that, if $w$ is a harmonic mapping  of the unit
disk into itself, then there hold the following sharp inequality
\begin{equation}\label{ona}\beta_w\le \frac{4}{\pi}.\end{equation}
See also the book of Pavlovi\'c \cite[p.~53,54]{lib} for a related
problem.

An estimates similar to \eqref{ona} for magnitudes of derivatives of
bounded harmonic functions in the unit ball in $\mathbf R^3$ is
obtained by Khavinson in \cite{kha}.

Together with the Bloch constants, for a harmonic mapping of the
unit disk onto itself consider the \emph{hyperbolic Lipschitz
constant} defined by
$$\beta_w^{\mathrm{hyp}}:=\sup_{z\neq
z'}\frac{d_h(w(z),w(z'))}{d_h(z,z')}.$$ Since
$|dz|\le{|dz|}/(1-|z|^2),$ it follows that for $z,w\in \mathbf U$ we
have $d(z,w)\le d_h(z,w).$  Thus $$\beta_w\le
\beta_w^{\mathrm{hyp}}.$$ It follows by Schwarz-Pick lemma that, if
$w$ is an analytic function then $$\beta_w^{\mathrm{hyp}}\le 1,$$
and the equality is attained for M\"obius selfmappings  of the unit
disk. Very recently it is proved in \cite{kavu} that, for real
harmonic mappings of the unit disk onto itself there hold the
following sharp inequality \begin{equation}\label{pro}|\nabla w|\le
\frac{4}{\pi}\frac{1-|w(z)|^2}{1-|z|^2},\end{equation} and therefore
$\beta_w^{\mathrm{hyp}}\le \frac{4}{\pi}$ extending thus Colonna
result for real harmonic mappings. However if we drop the assumption
that $w$ is real, then $\beta_w^{\mathrm{hyp}}$ can be infinite. The
inequality \eqref{pro} can be considered as a real-part theorem  for
an analytic function. More than one approach can be found in the
book \cite{real}.

In this paper we prove the following results for the unit disk which
are analogous to the results of Maz'ya \& Kresin  and extend the
results of Colonna by proving the following theorems.

Since the case $p=1$ is well-known, we will assume in the sequel
that $p>1$.

\begin{theorem}[Main theorem]\label{condi}
Let $p> 1$ and let $q$ be its conjugate. Let $w\in h^p$ be a complex
harmonic function defined in the unit disk and let $z\neq 0$. Define
$\mathbf n=\frac{z}{|z|}$, and $\mathbf t=i\frac{z}{|z|}$.

a) We have the following sharp inequalities
\begin{equation}\label{dl}|D w(z)e^{i\tau}|\le C_p(z,e^{i\tau})
(1-r^2)^{-1/p-1}\|w\|_{h^p},\end{equation}

\begin{equation}\label{dd}|D w(z)|\le C_p(z)
(1-r^2)^{-1/p-1}\|w\|_{h^p},\end{equation} where $z=re^{i\alpha}$,
\[\begin{split}
 C_p(z,e^{i\tau}) &= \frac{1}{\pi}\left(\int_{-\pi}^{\pi}\frac{\abs{\cos(s+\tau-\alpha)^q}}{(1+r^2-2r \cos
s)^{1-q}}ds\right)^{1/q}\end{split}\] and
\begin{equation}\label{po}C_p(z)=\left\{
                 \begin{array}{ll}
                    C_p(z,\mathbf n), & \hbox{if $p<2$;} \\
                    C_p(z,\mathbf t), & \hbox{if $p\ge 2$.}
                 \end{array}
               \right. \end{equation}
Moreover
         \begin{equation}\label{ara}\left\{
                 \begin{array}{ll}
                  C_p(z,\mathbf t)\le C_p(z,e^{i\tau})\le  C_p(z,\mathbf n), & \hbox{if $p<2$;} \\
                   C_p(z,\mathbf n)\le C_p(z,e^{i\tau})\le  C_p(z,\mathbf t), & \hbox{if $p\ge 2$.}
                 \end{array}
               \right. \end{equation}

b)  For $p\ge 2$ the function $C_p(z)$ can be expressed as
\begin{equation}\label{kre}
C_p(z)= \frac{2^{1/q}}{\pi} \left(\mathrm{B}\left(\frac{1 + q}{2},
\frac{1}{2}\right) F\left({1 - \frac{3 q}{2}, 1 - q}; {1 + \frac
q2};
    r^2\right)\right)^{1/q},
\end{equation}
where $\mathrm{B}$ is the beta function and $F$ is the Gauss
hypergeometric function.

c) Finally $$C_p:=\sup_{z\in \mathbf U}C_p(z)=\left\{
        \begin{array}{ll}\frac{1}{\pi}\left(\int_{-\pi}^{\pi}\frac{\abs{\cos
s}^q}{(2-2 \cos s)^{1-q}}ds\right)^{1/q}, & \hbox{if $1< p <2$;}
          \\
\frac{1}{\pi}\left(\int_{-\pi}^{\pi}\frac{\abs{\sin s}^q}{(2-2 \cos
s)^{1-q}}ds\right)^{1/q}, & \hbox{if $ p\ge2$.}
        \end{array}
      \right.$$ The constant $C_p$ is optimal
for real harmonic functions as well.
\end{theorem}

\begin{theorem}\label{two}
Let $p> 1$ and let $w\in h^p$, be a complex harmonic function
defined in the unit disk. Then we have the following sharp
inequalities
$$|\partial w(z)|, |\bar \partial w(z)|\le c_p(z) (1-|z|^2)^{-1/p-1}\|w\|_{h^p},$$
and
$$|\partial w(z)|, |\bar \partial w(z)|\le c_p (1-|z|^2)^{-1/p-1}\|w\|_{h^p},$$
where
\begin{equation}\label{cip}c_p(z)=(2\pi)^{1/q-1}( F(1 - q, 1 - q; 1;
r^2))^{1/q}\end{equation} and
\begin{equation}\label{cipir}c_p = 2^{\frac{-1 + q}{q}} \pi^{-1 + \frac{1}{2
q}}\left(\frac{\Gamma(-1/2 +
q)}{\Gamma(q)}\right)^{1/q}.\end{equation}
\end{theorem}

\begin{remark} a) In particular, if in Theorem~\ref{condi}
 we take $p=2$,
then we have the following estimate
\begin{equation}\label{nab}|\nabla w(z)|\leq
\frac{1}{\sqrt{\pi}}\frac{(1+|z|^2)^{1/2}}{(1-|z|^2)^{3/2}}\|w\|_{h^2}.\end{equation}
If we assume $w$ is a real harmonic function, i.e. $w=g+\overline
g$, where $g$ is an analytic function, then this estimate is
equivalent to the real part theorem
\begin{equation}\label{part}|g'(z)|\leq
\frac{1}{\sqrt{\pi}}\frac{(1+|z|^2)^{1/2}}{(1-|z|^2)^{3/2}}\|\Re
g\|_{h^2}.\end{equation} For the proof of \eqref{part} we refer to
\cite[pp. 87, 88]{real}. See also a higher dimensional
generalization of \eqref{nab} by Maz'ya and Kresin in the recent
paper \cite[Corollary~3]{Mar1} for $n\ge 2$. Also the relation
\eqref{nab} for real $w$ can be deduced from work of Macintyre and
Rogosinski for analytic functions, see \cite[p.~301]{maro}.
%

b) On the other hand if take $p=\infty$, then $C_p=\frac{4}{\pi}$
and therefore the relation \eqref{dd} coincides  with the result of
Colonna. While for real $w$, it is a real part theorem (\cite{kha})
which can be expressed as \begin{equation}\label{repart}|g'(z)|\leq
\frac{4}{\pi}\frac{1}{1-|z|^2}\|\Re g\|_{\infty}.\end{equation}

c)  Notice that $c_p<C_p<2c_p,$ if $p>1,$ and $C_1 =
2c_1=\frac{4}{\pi}$. On the other hand $c_\infty=1$ coincides with
the constant of the Schwarz lemma for analytic functions. Notice
also this interesting fact, the minimum of constants $C_p$ is
achieved for $p=2$ and is equal to $C_2=\sqrt{2/\pi}$. The graphs of
functions $C_p$ and $c_p$, $1\le p\le 20$ are shown in Figure~1 and
Figure~2.

d) From Theorem~\ref{condi} we find out that, the Khavinson
hypothesis (see \cite{Mar1}) is not true for $n=2$ and $2<p<\infty$.
Namely the maximum of the absolute value of the directional
derivative of a harmonic function with a fixed $L^p$-norm of its
boundary values is attained at the radial direction for $p\le 2$ and
at the tangential direction for $2<p<\infty$.
\end{remark}

\begin{figure}[htp]\label{lako}
\centering
\includegraphics{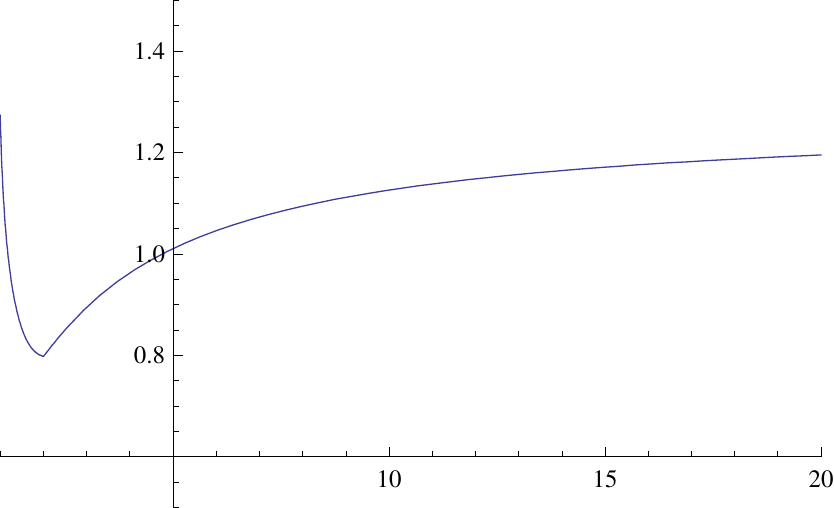}
\caption{The graph of  $C_p$.}
\end{figure}

\begin{figure}[htp]\label{la}
\centering
\includegraphics{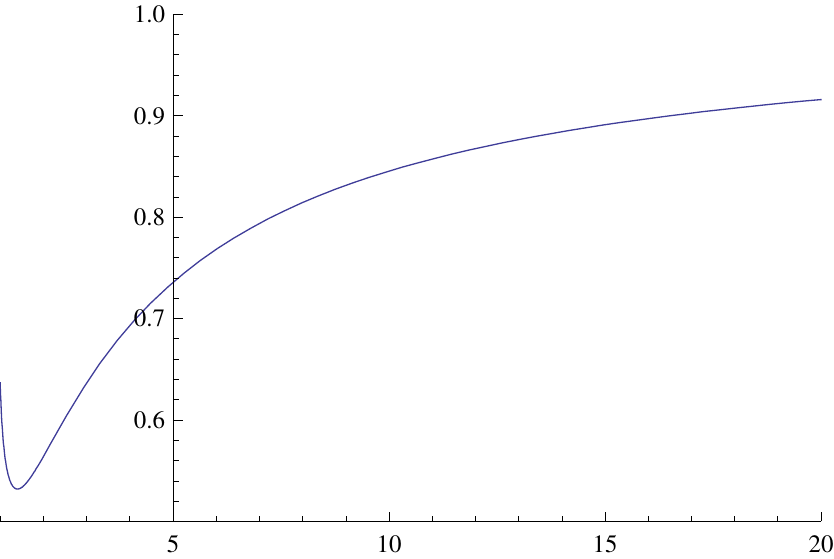}
\caption{The graph of  $c_p$.}
\end{figure}
In the classical paper  \cite[(8.3.8)]{maro} of Macintyre and
Rogosinski they obtained the inequality \begin{equation}
\label{primo}|f'(z)|\le
\left(1+\frac{r^2}{(p-1)^2}\right)^{1/q}(1-|z|^2)^{-1-1/p}\|f\|_{H^p}.\end{equation}
In the following direct corollary of Theorem~\ref{two} we improve
the inequality \eqref{primo} by proving

\begin{corollary}\label{pol}
Let $w=f(z)$ be an analytic function from the Hardy class
$H^p(\mathbf U)$. Then there hold the following inequality
\begin{equation}\label{prim}|f'(z)|\le c_p(z)(1-|z|^2)^{-1-1/p}\|f\|_{H^p},\end{equation} where
$c_p(z)$ is defined in \eqref{cip}.
\end{corollary}

\begin{remark}
Corollary~\ref{pol} is an improvement of corresponding inequality
\cite[(8.3.8)]{maro} because
$$(2\pi)^{1-q} F(1 - q, 1 - q; 1; r^2)< 1+\frac{r^2}{(p-1)^2}$$ for
all $q>1$. 
%
It is not known by the authors if all functions $c_p(z)$ of
Corollary~\ref{pol} are sharp, however the function
$c_2(z)=\frac{\sqrt{1 + |z|^2}}{\sqrt{2 \pi}}$ is sharp, because
\eqref{prim} coincides with the sharp inequality \cite[p.~301,
eq.~(7.2.1)]{maro} for $p=2$. On the other hand the power $-1-1/p$
is optimal see e.g. Garnett \cite[p.~86]{gar}.  The paper
\cite{maro} contains some sharp estimates $|f^{(k)}(z)|\le
c_p\|f\|_p$ for $f\in H^p(\mathbf U)$ and $k\ge 1$ but $p$ depends
on $k$ and it seems that if $k=1$ then $p$ can be only 1 or 2.
\end{remark}

\section{Proofs}
We need the following lemmas
\begin{lemma}\label{male}
Let $a_q(t)$, $t\in[0,2\pi]$, $q\ge 1$, $0\leq r\leq1$ be a function
defined by
$$a_q(t)= \int_{-\pi}^{\pi}\abs{\cos(s-t)}^q|r-e^{is
}|^{2q-2}ds.$$ Then $$\max_{0\le t\le 2\pi}a_q(t)=\left\{
                                                     \begin{array}{ll}
                                                      a_q(\frac{\pi}{2}) , & \hbox{if $q\le2$;} \\
                                                       a_q(0), & \hbox{if $q>2$.}
                                                     \end{array}
                                                   \right. $$
                                                   and
                                                   $$\min_{0\le t\le 2\pi}a_q(t)=\left\{
                                                     \begin{array}{ll}
                                                      a_q(0) , & \hbox{if $q\le2$;} \\
                                                       a_q(\frac{\pi}{2}), & \hbox{if $q>2$.}
                                                     \end{array}
                                                   \right. $$
\end{lemma}
\begin{proof} Since $q=1$ is trivial
let $q>1$ and
$$a(t):=a_q(t)=\int_{-\pi}^{\pi}\abs{\cos (t-s)}^q(1+r^2-2r\cos s)^{q-1}ds.$$
Note that $a$ is $\pi-$periodic. Because sub-integral expression is
$2\pi-$periodic with respect to $s$ we obtain
$$a(t)=\int_{0}^{2\pi}\abs{\cos s}^q(1+r^2-2r\cos (t+s))^{q-1}ds,$$
and therefore
$$a^{\prime}(t)=2(q-1)r\int_{0}^{2\pi}\abs{\cos s}^q \sin (t+s) (1+r^2-2r\cos(t+s))^{q-2}ds.$$
Again by using the periodicity of sub-integral expression
$$a^{\prime}(t)=2(q-1)r\int_{0}^{2\pi}\abs{\cos (t-s)}^q \sin s (1+r^2-2r\cos s)^{q-2}ds.$$
Next we need some transformations
\[\begin{split} a^{\prime}(t)&=2(q-1)r\int_{0}^{\pi}\abs{\cos
(t-s)}^q \sin s (1+r^2- 2r\cos s)^{q-2}ds\\&
+2(q-1)r\int_{\pi}^{2\pi}\abs{\cos (t-s-\pi)}^q \sin (s+\pi) (1+r^2-
2r\cos (s+\pi))^{q-2}ds\\& =2(q-1)r\int_{0}^{\pi}\abs{\cos (t-s)}^q
\sin s \, Q\left(r,s-{\pi}/{2}\right)ds
\\&=2(q-1)r\int_{-\pi/2}^{\pi/2}\abs{\sin (t-s)}^q \cos s\, Q(r,s)ds,\end{split}\] where $$Q(r,s)=(1+r^2+2r\sin
s)^{q-2}-(1+r^2-2r\sin s)^{q-2}.$$ Thus the derivative is
$$a^{\prime}(t)=2r(q-1)\int_{-\pi/2}^{\pi/2} h(t,s) \cos s ds,$$
where
$$h(t,s)=\abs{\sin (t-s)}^q Q(r,s).$$
Also $a^{\prime}(t)$ is $\pi-$periodic and
$$a^{\prime}(0)=a^{\prime}(\pi /2)=0.$$
Further
$$h(t,s)+h(t,-s)=(\abs{\sin (t-s)}^q -\abs{\sin (t+s)}^q )Q(r,s).$$
If  $1<q<2$, then for $0<t<\pi/2$ we have
$$h(t,s)+h(t,-s)>0,\ \  0<s<\pi/2$$
and  $\pi/2<t<\pi$
$$h(t,s)+h(t,-s)>0,\ \  0<s<\pi/2.$$
We claim that
$$ a^{\prime}(t)=2r(q-1)\int_{0}^{\pi/2} (h(t,s)+h(t,-s)) \cos s ds>0,\ \  0<t<\pi/2 $$
and
$$ a^{\prime}(t)=2r(q-1)\int_{0}^{\pi/2} (h(t,s)+h(t,-s)) \cos s ds<0, \pi/2<t<\pi. $$
It means that the minimum of $a$ is achieved in $0$ and  the maximum
in $\frac{\pi}2$.

Similarly can be treated the case $q>2$. For  $q=2$ the function
$a(t)$ is a constant. The proof of Lemma~\ref{male} is completed.
\end{proof}

\begin{lemma}\label{polet}
Let $\lambda\ge 0$, $0\le r\le 1$ and $q\geq1$. For all $t$ there
exists $t'\in [0,2\pi]$ such that
$$\int_{0}^{2\pi} |\cos(s-t)|^\lambda |r-e^{is}|^{2q-2} ds
\le \int_{0}^{2\pi} |\cos(s-t')|^\lambda |1-e^{is}|^{2q-2}ds.
$$
\end{lemma}
\begin{proof} In order to prove Lemma~\ref{polet}, we need the following proposition.
\begin{proposition}\cite[Lemma~3.2]{anali}\label{lema}
Let $\mathbf U\subset \mathbf C$ be the open unit disk and $(A,\mu)$
be a measured space with $\mu(A)<\infty$. Let $f(z,\omega)$ be a
holomorphic function for $z\in \Bbb U$ and measurable for $\omega\in
A$. Let $b>0$ and assume in addition that, there exists an
integrable function $\chi\in L^{\max\{b,2\}}(A,d\mu)$ such that
\begin{equation}\label{pavlovic}|f(0,\omega)|+ |f'(z,\omega)|\le \chi(\omega),\end{equation} for
$(z,\omega)\in \Bbb U \times A$, where by $f'(z,\omega)$ we mean the
complex derivative of $f$ with respect to $z$. Then the function
$$\phi(z) = \log\int_{A} |f(z,\omega)|^ bd\mu(\omega)$$ is
subharmonic in $\Bbb U$.
\end{proposition}

\begin{corollary}\label{copo}
Assume together with the assumptions of the previous proposition
that $z\to f(z,\omega)$ is continuous up to the boundary $\mathbf
T$. Then we have the following inequality
$$\phi(z)\le \max_{\tau \in[0,2\pi)}\phi(e^{i\tau})=\phi(e^{i\tau'}).$$
\end{corollary}
In order to apply Corollary~\ref{copo}, we take
$$d\mu(s)=|\cos(s-t)|^\lambda ds, \ \ \ f(z,s)=z-e^{is} \text{ and
}b=2q-2$$ and observe that \[\begin{split}\max_{\tau}\int_{0}^{2\pi}
|\cos(s&-t)|^\lambda |e^{i\tau}-e^{is}|^{2q-2}ds\\&=\int_{0}^{2\pi}
|\cos(s-t)|^\lambda |e^{i\tau'}-e^{is}|^{2q-2}ds\\&=\int_{0}^{2\pi}
|\cos(s-t')|^\lambda |1-e^{is}|^{2q-2}ds.\end{split}\] This finishes
the proof of Lemma~\ref{polet}.
\end{proof}
The Poisson kernel for the disc can be expressed as
$$
P(z,e^{\theta})=\frac{1-|z|^2}{|z-e^{i\theta}|^2}=-\left(1+
\frac{e^{-i\theta}}{\overline z -
  e^{-i\theta}}+\frac{e^{i\theta}}{z-e^{i\theta}}\right).$$ Then we
have
$$\mathrm{grad}(P) =(P_x,P_y)=P_x+iP_y=2\bar\partial P=\frac{2e^{-i\theta}}{(\overline z -
  e^{-i\theta})^2},$$
$$\partial P =
\frac{e^{i\theta}}{\left({ z-e^{i\theta}}\right)^2}$$ and
$$\bar \partial P =
\frac{e^{-i\theta}}{\left({\bar z-e^{-i\theta}}\right)^2}.$$
\begin{proof}[Proof of Theorem~\ref{condi}] a)
 Let $l=e^{i\tau}$. Then for $p>1$
\begin{equation}\label{cas}\begin{split}D w(z) l &=
\frac{1}{2\pi}\int_0^{2\pi}\left<\mathrm{grad}(P),l\right>f(e^{i\theta})
{d\theta}
\\&=\frac{1}{\pi}\int_0^{2\pi}\Re\frac{e^{-i(\theta+\tau)}}{\left({\bar
z-e^{-i\theta}}\right)^2}f(e^{i\theta})
{d\theta}.\end{split}\end{equation}
\\
We apply \eqref{cas} and H\"older inequality in order to obtain
$$|D w(z) l|\le \frac{1}{\pi}\left(\int_0^{2\pi}|\Re\frac{e^{-i(\theta+\tau)}}{\left({\bar
z-e^{-i\theta}}\right)^2}|^qd\theta\right)^{1/q}\left(\int_0^{2\pi}|f(e^{i\theta})|^p
{d\theta}\right)^{1/p}.$$ We should consider the integral
$$I_q=\int_0^{2\pi}\abs{\Re\frac{e^{-i(\theta+\tau)}}{\left({
\bar z-e^{-i\theta}}\right)^2}}^q d\theta.$$ First of all
$$I_q=\int_0^{2\pi}\abs{\Re\frac{e^{i(\theta+\tau)}}{\left({
z-e^{i\theta}}\right)^2}}^q d\theta
=\int_0^{2\pi}\abs{\Re\frac{e^{i(\theta + \tau-\alpha)}}{\left({
 r-e^{i\theta}}\right)^2}}^q d\theta. $$
Take the substitution
$$e^{i\theta}=
\frac{r-e^{is}}{1- re^{is}}.$$ Then $$de^{i\theta}=\frac{1-r^2}{(1-
re^{is})^2}de^{is},$$ and thus
\[\begin{split}d\theta&=\frac{1-r^2}{(1- re^{is})^2}\frac{e^{is}}{e^{i\theta}}ds\\&=e^{is}\frac{1-r^2}{(1- re^{is})^2}
\frac{1- re^{is}}{ r-e^{is}}ds\\&=\frac{1-r^2}{1+r^2-2r\cos s}ds.
\end{split}\]
On the other hand, we easily find that
$$\Re\frac{e^{i(\theta+\tau-\alpha)}}{\left({
r-e^{i\theta}}\right)^2}=\frac{(1+r^2-2r\cos
s)\cos(s+\tau-\alpha)}{(1-r^2)^2}.$$ Therefore, finally we have the
relation
\begin{equation}\label{ere}\begin{split}\int_0^{2\pi}\abs{\Re\frac{e^{-i(\theta+\tau)}}{\left({\bar
z-e^{-i\theta}}\right)^2}}^q
d\theta&=(1-|z|^2)^{1-2q}\int_{-\pi}^{\pi}\frac{\abs{\cos(s+\tau-\alpha)}^q}{(1+r^2-2r\cos
s)^{1-q}}ds,
\end{split}\end{equation}
which together with first relation give
$$|Dw(z)l|\leq C_p(z,l)(1-|z|^2)^{-1-1/p}\|w\|_{h^p}.$$
Now by using Lemma~\ref{male} we conclude that
\begin{equation*}C_p(z)=\left\{
                 \begin{array}{ll}
                    C_p(z,\mathbf n), & \hbox{if $p<2$;} \\
                    C_p(z,\mathbf t), & \hbox{if $p\ge 2$,}
                 \end{array}
               \right. \end{equation*} which coincides with
               \eqref{po}. This implies \eqref{dd}. Lemma~\ref{male}
               implies at once \eqref{ara}.

%
b) By using the following formula
\begin{equation}\label{formula}\int_0^{\pi} \frac{\sin^{\mu-1}t}{(1+r^2-2
r \cos t)^{\nu}}
dt=\mathrm{B}\left(\frac{\mu}{2},\frac{1}{2}\right)F\left(\nu,\nu
+\frac{1-\mu}{2};\frac{1+\mu}{2},r^2\right)\end{equation} (see,
e.g., Prudnikov, Brychkov and Marichev \cite[2.5.16(43)]{pbm}),
where $\mathrm{B}(u, v)$ is the Beta-function, and $F(a, b; c; x)$
is the hypergeometric Gauss function, for $\mu = q+1$ and $\nu =
1-q$, because $\abs{\cos(s+\tau-\alpha)}^q=\abs{\sin s}^q$, for
$\tau =\alpha+\frac{\pi}{2}$, we obtain \eqref{kre}.

c) By using both Lemma~\ref{polet} and Lemma~\ref{male} we obtain:
$$C_p(z,l)\leq C_p(1,l')\leq C_p$$
for some $l', |l'|=1$ and we have second conclusion of main theorem.

Let us now show that the constant $C_p$ is sharp.  We will show the
sharpness of the result for $p\le 2$. A similar analysis works for
$p>  2$. Let $0<\rho<1$ and take
$$e^{is}=\frac{\rho-e^{it}}{1-\rho e^{it}},$$ i.e.
$$e^{it}=\frac{\rho-e^{is}}{1-\rho e^{is}}.$$ Define
$$f_\rho(e^{it})=(1-\rho^2)^{-1/p}\abs{\cos s (1-\cos s)}^{q-1}\sign(\cos s).$$ And
take $$w_\rho=P[f_\rho].$$ Then
$$dt = \frac{1-\rho^2}{1+\rho^2-2\rho \cos s}ds,$$ $$\Re\frac{e^{it}}{\left({
r-e^{it}}\right)^2}=\frac{(1+r^2-2r\cos s)\cos(s)}{(1-r^2)^2},$$ and
\[\begin{split}\int_{0}^{2\pi}|f_\rho(e^{it})|^pdt &=\int_{0}^{2\pi}|f_\rho(e^{it})|^p
\frac{1}{1+\rho^2-2\rho \cos s}ds\\&=\int_{0}^{2\pi}\abs{\cos s
(1-\cos s)}^{q}\frac{1}{1+\rho^2-2\rho \cos s}ds.\end{split}\] Thus
\begin{equation}\label{ewo}\lim_{\rho \to 1}\|f_\rho\|_p^p= \int_{0}^{2\pi}|f_\rho(e^{it})|^pdt =
\frac{\pi^q}{2^{q}}C_p^q.\end{equation} By taking $r=\rho$, we
obtain
\[\begin{split}(1-\rho^2)^{1+1/p}|D w_\rho(\rho) 1|&=
\frac{(1-\rho^2)^{1+1/p}}{\pi}\int_0^{2\pi}\Re\frac{e^{i t}}{\left({
\rho-e^{i\theta}}\right)^2}f_\rho(e^{it})
{dt}\\&=\frac{(1-\rho^2)^{1+1/p}}{\pi}\int_0^{2\pi}
\frac{(1+\rho^2-2\rho\cos s)\cos(s)} {(1-\rho^2)^2}
(1-\rho^2)^{-1/p}\\&\times\abs{\cos s (1-\cos s)}^{q-1}\sign(\cos s)
\frac{1-\rho^2}{1+\rho^2-2\rho \cos
s}ds\\&=\frac{1}{\pi}\int_{0}^{2\pi}\abs{\cos s}^{q} (1-\cos
s)^{q-1}ds\\&=\frac{\pi^{q-1}}{2^{q-1}}C_p^q\end{split}\] From
\eqref{ewo} it follows that
$$\lim_{\rho\to 1}\frac{(1-\rho^2)^{1+1/p}|D w_\rho(\rho)
1|}{\|f_\rho\|_p}=C_p.$$  This shows that the constant $C_p$ is
sharp.
\end{proof}

\begin{proof}[Proof of Theorem~\ref{two}]
First of all
$$\partial w =
\int_{0}^{2\pi}\frac{e^{i\theta}}{\left({
z-e^{i\theta}}\right)^2}f(e^{i\theta})\frac{d\theta}{2\pi}.$$ By
applying H\"older inequality we have
\[\begin{split}|\partial w|&\le \frac
1{2\pi}\left(\int_{0}^{2\pi}\frac{1}{\left|{
z-e^{i\theta}}\right|^{2q}}d\theta\right)^{1/q}\left(\int_0^{2\pi}|f(e^{i\theta})|^p{dt}\right)^{1/p}\\&=(1-|z|^2)^{1/q-2}\frac
1{2\pi}\left(\int_{0}^{2\pi}\frac{(1-|z|^2)^{2q-1}}{\left|{
z-e^{i\theta}}\right|^{2q}}d\theta\right)^{1/q}\left(\int_0^{2\pi}|f(e^{i\theta})|^p{dt}\right)^{1/p}.\end{split}\]
It remains to estimate the integral
$$J_q= \int_{0}^{2\pi}\frac{(1-|z|^2)^{2q-1}}{\left|{
z-e^{i\theta}}\right|^{2q}}d\theta
=\int_{0}^{2\pi}\frac{(1-r^2)^{2q-1}}{|r-e^{i\theta}|^{2q}}d\theta.$$
By making use again of the change
$$ e^{i\theta}
=\frac{r-e^{is}}{1-re^{is}},$$ we obtain $$d\theta =
\frac{1-r^2}{|1-re^{is}|^2}ds$$ and  $$r-e^{i\theta}
=\frac{(1-r^2)e^{is}}{1-re^{is}}.$$ Therefore by using
Lemma~\ref{polet} for $\lambda=0$ we obtain
\[\begin{split}J_q=\int_{0}^{2\pi}\frac{(1-r^2)^{2q-1}}{|r-e^{i\theta}|^{2q}}d\theta&=(1-r^2)^{1-q}\int_{0}^{2\pi}|1-re^{is}|^{2q-2}ds
\\&=(1-r^2)^{1-q}\int_{0}^{2\pi}|1+r^2-2r\cos s|^{q-1}ds\\&
\leq 2^{q-1}(1-r^2)^{1-q}\int_{0}^{2\pi}|1-\cos
s|^{q-1}ds.\end{split}\] Thus
$$|\partial w|\le c_p (1-|z|^2)^{-1-1/p}\|f\|_{L^{p}(\mathbf T)},$$
where
$$c_p= 2^{\frac{-1 + q}{q}} \pi^{-1 + \frac{1}{2 q}}\left(\frac{\Gamma(-1/2 +
q)}{\Gamma(q)}\right)^{1/q}.$$ This proves \eqref{cipir}. By formula
\eqref{formula} for $\mu =1$, $\nu =1-q$ we have
\[\begin{split}\int_{0}^{2\pi}|1+r^2-2r\cos
s|^{q-1}ds&=2\int_{0}^{\pi}|1+r^2-2r\cos s|^{q-1}ds\\&=2\pi
F\left(1-q,1-q ;1,r^2\right).\end{split}\] This implies \eqref{cip}.
The sharpness of constant $c_p$ can be verified by taking
$$f^{\pm}_\rho(e^{it})=(1-\rho^2)^{-1/p}\abs{\cos s (1-\cos
s)}^{q-1}e^{\pm is}$$ and following the proof of sharpness of $C_p$.
\end{proof}

\subsection{Acknowledgement} After we wrote the first version of
this paper, we had useful discussion about this subject with
professor Vladimir Maz'ya.


\begin{thebibliography}{99}
\bibitem{ABR}
\textsc{S. Axler, P. Bourdon and W. Ramey:} {\it Harmonic function
theory}, Springer Verlag New York 1992.
  \bibitem{cor}
\textsc{F. Colonna:} {\it The Bloch constant of bounded harmonic
mappings.} Indiana Univ. Math. J. \textbf{38} (1989), no. 4,
829--840.

\bibitem{gar} \textsc{J. Garnett:}  {\it Bounded analytic
functions.} Pure and Applied Mathematics, 96. Academic Press, Inc.,
New York-London, 1981. xvi+467 pp.
\bibitem{kha}
\textsc{D. Khavinson:} {\it An extremal problem for harmonic
functions in the ball}, Canad. Math. Bull., \textbf{35} (1992),
218-220.



\bibitem{kavu}
\textsc{D. Kalaj and M. Vuorinen:} {\it On harmonic functions and
the Schwarz lemma,} to appear in Proceedings of the AMS.


\bibitem{anali} \textsc{D. Kalaj:} {\it On isoperimetric inequality for the polydisk,} to appear in Annali di matematica pura ed applicata, DOI:10.1007/s10231-010-0153-2.
\bibitem{Mar}
\textsc{G. Kresin and V. Maz'ya:}  {\it Optimal estimates for the
gradient of harmonic functions in the multidimensional half-space.}
Discrete Contin. Dyn. Syst. \textbf{28} (2010), no. 2, 425--440.

\bibitem{Mar1}
\textsc{G. Kresin and V. Maz'ya:} {\it Sharp pointwise estimates for
directional derivatives of harmonic functions in a multidimensional
ball.} Journal of Mathematical Sciences, \textbf{169}, No. 2, 2010.


\bibitem{real}
\textsc{G. Kresin and V. Maz'ya:} {\it Sharp Real-Part Theorems A
Unified Approach,} Lecture Notes in Mathematics, Volume 1903, 2007.
\bibitem{maro}
\textsc{A. J. Macintyre and W. W. Rogosinski:} {\it Extremum
problems in the theory of analytic functions}, Acta Math.
\textbf{82}, 1950, 275 - 325.

\bibitem{pbm} \textsc{A. P. Prudnikov, Yu. A. Brychkov and O. I. Marichev}, {\it Integrals
and Series}, Vol. 1: Elementary Functions. Gordon and Breach Sci.
Publ., New York, 1986.

\bibitem{lib}
\textsc{M. Pavlovi\'c:} {\it Introduction to function spaces on the
disk.} 20. Matemati\v cki Institut SANU, Belgrade, 2004. vi+184 pp.
\end{thebibliography}
\end{document}